\documentclass[10pt]{amsart}

\usepackage{amssymb,amsmath,amsthm}
\usepackage[arc,all]{xy}

\usepackage{enumitem}
\usepackage{url}
\usepackage{scrextend}

\usepackage{enumitem}

\theoremstyle{plain}
\newtheorem{cor}[subsection]{Corollary}
\newtheorem{lem}[subsection]{Lemma}
\newtheorem{prop}[subsection]{Proposition}
\newtheorem{thm}[subsection]{Theorem}

\newtheorem{assumption}[subsection]{Assumption}

\theoremstyle{definition}
\newtheorem{defn}[subsection]{Definition}

\theoremstyle{remark}
\newtheorem{rem}[subsection]{Remark}

\newcommand{\Ho}{{ \mathsf{Ho} }}

\newcommand{\Spectra}{{ \mathsf{Sp}^\Sigma }}

\newcommand{\Alg}{{ \mathsf{Alg} }}

\newcommand{\TQ}{{ \mathsf{TQ} }}

\newcommand{\AlgO}{{ \Alg_\capO }}

\newcommand{\capO}{{ \mathcal{O} }}

\newcommand{\id}{{ \mathrm{id} }}

\newcommand{\iso}{{ \cong }}

\DeclareMathOperator*{\hocolim}{hocolim}

\DeclareMathOperator{\BAR}{Bar}

\newcommand{\abs}[1]{\left\vert#1\right\vert}
\newcommand{\sett}[1]{\left\{#1\right\}}
\newcommand\weto{\stackrel{\sim}{\smash{\longrightarrow}\rule{0pt}{0.5ex}}}

\title[Chromatic localization and homotopy completion]{On the chromatic localization of the homotopy completion tower for $\capO$-algebras}

\author{Crichton Ogle}
\address{Department of Mathematics, The Ohio State University, 231 West 18th Ave, Columbus, OH 43210, USA}
\email{ogle.1@osu.edu}

\author{Nikolas Schonsheck}
\address{Department of Mathematical Sciences, University of Delaware, 501 Ewing Hall, Newark, DE 19716, USA}
\email{nischon@udel.edu}

\begin{document}

\maketitle

\begin{abstract}
The completion tower of a nonunital commutative ring is a classical construction in commutative algebra. In the setting of structured ring spectra as modeled by algebras over a spectral operad, the analogous construction is the homotopy completion tower. The purpose of this brief note is to show that smashing localizations, such as that given by the Johnson-Wilson spectrum $E(n)$, commute with the terms of this tower.
\end{abstract}

\section{Introduction}
We work in the context of symmetric spectra \cite{Hovey_Shipley_Smith}, and consider any algebraic structure described by algebras over a reduced operad $\capO$; that is, $\capO[0]=\ast$ and hence $\capO$-algebras are nonunital . (See \cite{EKMM} and \cite{Mandell_May_Schwede_Shipley} for other well-behaved categories of spectra.) To any $\capO$-algebra $X$, there is an associated homotopy completion tower \cite{Harper_Hess}, analogous to the $R$-adic completion tower of a nonunital commutative ring $R$. The aim of this short paper is to show that the terms of this tower commute with smashing localizations. 

To keep this paper appropriately concise, we freely use the notation of \cite{Harper_Hess} in discussing homotopy completion in $\AlgO$. The basic idea of the construction is that the map of operads $\capO \to \tau_k\capO$ induces a Quillen adjunction
\begin{align}\label{eq:barq_baru}
\xymatrix{
\AlgO \ar@<0.5ex>^-{\tau_k\capO\circ_\capO-}[r] & \Alg_{\tau_k\capO}\ar@<0.5ex>^-{F}[l]
}
\end{align}
where $\tau_k\capO$ is the degree $k$ truncation of $\capO$, i.e., $\tau_k\capO$ is equal to $\capO$ in degrees less than or equal to $k$ and is trivial above. For any $\capO$-algebra $X$, this construction gives rise to a tower $X\to \sett{\tau_k\capO\circ_\capO(X)}_k$. If $X$ is cofibrant, this construction is homotopical and is known as the homotopy completion tower of $X$. In fact, as we elaborate in Section \ref{sec_identification_taylor_tower}, this tower can be identified with the Taylor tower \cite{Goodwillie_calculus_3} of the identity functor on $\AlgO$. Our discussion of this identification supplements \cite[2.21]{Kuhn_Pereira} and \cite[4.3]{Pereira_spectral_operad}.

Throughout the paper, we let $E$ be a symmetric spectrum with the following property. For $\mathcal{U} \colon \mathsf{Sp}^\Sigma \to \mathsf{Sp}^\mathbb{N}$ the forgetful functor from symmetric to ordinary (or, sequential) spectra, assume that the localization functor associated to $\mathcal{U}E$, as defined in \cite{Bousfield_localization_spectra}, is smashing. We will denote this localization by $L_E$. Much of our technical work involves constructing a well-behaved model $S_E$ of the $E$-local sphere spectrum (see Theorem \ref{thm_main_lemma_2}). Using this model, we obtain our main result, proven in Section \ref{sec_convenient_model}. Note that, below, the superscript ``$\mathsf{h}$'' is added to denote a suitably derived version of $\tau_n\capO\circ_\capO(-)$, as in \cite[3.3]{Harper_Hess}; in particular, $\tau_n\capO\circ_\capO^\mathsf{h}(X)$ is fibrant for all $X$. We also remind the reader that equivalences in $\mathsf{Sp}^\mathbb{N}$ are strictly stronger than equivalences in $\mathsf{Sp}^\Sigma$, since the latter need not induce isomorphisms on stable homotopy groups.

\begin{thm}\label{thm_main_theorem}
Let $\capO$ be a $\Sigma$-cofibrant operad in $\Spectra$ and $\mathcal{U} \colon \mathsf{Sp}^\Sigma \to \mathsf{Sp}^\mathbb{N}$ the forgetful functor from symmetric to ordinary spectra. If $X$ is a fibrant $\capO$-algebra, then there is a weak equivalence
\begin{align}
L_E\big(\mathcal{U}\tau_n\capO\circ_\capO^\mathsf{h}(X)\big) \simeq \mathcal{U}\tau_n\capO\circ_\capO^\mathsf{h}(S_E \wedge X) 
\end{align}
with $\mathcal{U}(S_E \wedge X) \simeq L_E(\mathcal{U}X)$.
\end{thm}

\begin{rem}
The particular case of Theorem \ref{thm_main_theorem} that originally motivated this paper is obtained when one takes $E$ to be the Johnson-Wilson spectrum $E(n)$ for a fixed prime $p$; see, for instance, \cite{Johnson_Wilson}. A construction of $E(n)$ in symmetric spectra is given in \cite[I.6.63]{Schwede_book_project} and it is a classical result \cite{Ravenel_nilpotence} that the corresponding $E(n)$-localization is smashing. It is also worth noting that the first layer $\tau_1\capO\circ^\mathsf{h}_\capO(X)$ of the homotopy completion tower of an $\capO$-algebra $X$ is the $\TQ$-homology of $X$ (see, e.g., \cite{Ching_Harper_derived_Koszul_duality}, \cite{Harper_bar_constructions}, \cite{Schonsheck_fibration_theorems}, and \cite{Schonsheck_taylor_tower_identity}). Thus, under appropriate fibrancy conditions, Theorem \ref{thm_main_theorem} shows that $\TQ$-homology commutes with chromatic (or, more generally, smashing) localizations.
 \end{rem}

\begin{rem}\label{rem_sigma_cofibrancy_of_O}
Although we have assumed $\Sigma$-cofibrancy in Theorem \ref{thm_main_theorem}, any operad can be replaced by a weakly equivalent $\Sigma$-cofibrant operad, and the induced map of homotopy completion towers is a weak equivalence; see, for instance, \cite[3.26, 5.48]{Harper_Hess}.
\end{rem}

In Section 4, we show an equivalence between the homotopy completion tower of an $\capO$-algebra $X$ and the Taylor tower of the identity functor evaluated on $X$, with no connectivity assumptions on $X$. Specifically, Theorem \ref{thm_identification_of_taylor_tower} establishes a zigzag of weak equivalences
\begin{equation}
\tau_n\capO\circ_\capO(X) \simeq P_n(\id)X
\end{equation}
for all cofibrant $\capO$-algebras $X$. This result, combined with Theorem \ref{thm_main_theorem}, implies the following.

\begin{thm}\label{thm_main_corollary} 
Under the assumptions of Theorem \ref{thm_main_theorem}, if $X$ is furthermore cofibrant, then there is a zigzag of weak equivalences
\begin{equation}
L_E(\mathcal{U}P_n(\id)X) \simeq \mathcal{U}P_n(\id)(S_E\wedge X)
\end{equation}
with $\mathcal{U}(S_E \wedge X) \simeq L_E(\mathcal{U}X)$.
\end{thm}

Our main result, Theorem \ref{thm_main_theorem}, follows without too much difficulty from Theorems \ref{thm_main_lemma_1} and \ref{thm_main_lemma_2}, which we turn to now.

\begin{thm}\label{thm_main_lemma_1}
Let $R$ be a commutative ring spectrum, i.e., a commutative monoid in $\Spectra$, for which the pairing map $R \wedge R \weto R$ is a weak equivalence and suppose that $R$ is cofibrant as a symmetric spectrum. If $\capO$ is a $\Sigma$-cofibrant operad in $\Spectra$, then for any $\capO$-algebra $X$, there is a zigzag of weak equivalences 
\begin{align}\label{eqn_main_thm}
R \wedge (\tau_n\capO\circ_\capO^\mathsf{h} X) \simeq \tau_n\capO\circ_\capO^\mathsf{h}(R\wedge X)
\end{align}
in the underlying category $\Spectra$.
\end{thm}

The proof of Theorem \ref{thm_main_lemma_1} is the content of Section \ref{sec_proof_of_main_lemma_1}. In Section \ref{sec_convenient_model}, we construct $S_E$ and prove the following, along with Theorem \ref{thm_main_theorem}.

\begin{thm}\label{thm_main_lemma_2}
There exists a commutative ring spectrum $S_E$ with $S_E \wedge S_E \weto S_E$ a weak equivalence, and which is cofibrant as a symmetric spectrum, with the following property: for any $X \in \Spectra$, there is a natural isomorphism
\begin{align}\label{eqn_good_model_localization}
\mathcal{RU}(S_E\wedge X) \ \iso \ \mathcal{RU}(S_E)\wedge_{\mathcal{SHC}} \mathcal{RU}(X) \ \iso \ L_E(\mathcal{RU}X)
\end{align}
in the stable homotopy category, where $\mathcal{RU}$ denotes the right-derived functor of $\mathcal{U}$. 
\end{thm}

\begin{rem}
	In order to leverage certain compatibility results in \cite{Mandell_May_Schwede_Shipley}, notably relating the smash product in $\mathsf{Sp}^\Sigma$ and $\mathsf{Sp}^\mathbb{N}$, we have phrased Theorem \ref{thm_main_lemma_2}, its proof, and the proof of \ref{thm_main_theorem} in terms of ``the'' stable homotopy category. In order to be explicit, we take $\Ho(\mathsf{Sp}^\mathbb{N})$, the homotopy category of the Bousfield-Friedlander model category structure on sequential spectra $\mathsf{Sp}^\mathbb{N}$, as our model of the stable homotopy category. More concisely, we define $\mathcal{SHC} = \Ho(\mathsf{Sp}^\mathbb{N})$ and note that, in particular, this means the objects of $\mathsf{Sp}^\mathbb{N}$ and $\Ho(\mathsf{Sp}^\mathbb{N})$ are the same \cite[5.6]{Dwyer_Spalinski}. Consequently, if $X \in \mathsf{Sp}^\Sigma$ then $\mathcal{U}X$ lives, as an object, both in $\mathsf{Sp}^\mathbb{N}$ and $\mathcal{SHC}$. Similarly, since $\mathcal{RU}X$ is $\mathcal{U}$ applied to a fibrant replacement of $X$, the object $\mathcal{RU}X$ also lives both in $\mathsf{Sp}^\mathbb{N}$ and $\mathcal{SHC}$. For these same reasons, it is also appropriate to apply $L_E$ to any object in $\mathsf{Sp}^\mathbb{N}$. To distinguish between the smash product of $\mathsf{Sp}^\Sigma$ and that of $\mathcal{SHC}$, we denote the latter by $\wedge_{\mathcal{SHC}}$ and the former by simply $\wedge$.
\end{rem}

It is worth pointing out that, throughout this paper, we are considering the homotopy completion of $R\wedge X$ \emph{as an $\capO$-algebra}. It is known (see, e.g., \cite[2.1]{Kuhn_Pereira}) that the objectwise smash of an operad with a commutative ring spectrum is again an operad. Thus, one could also consider the homotopy completion of $R \wedge X$ as an algebra over the operad $R \wedge\capO$ which has values in $R$-modules (where the symmetric monoidal product is $\wedge_R$). In this case, the result analogous to Theorem \ref{thm_main_lemma_1} follows from \cite[4.10]{Harper_Hess} and \cite[2.11(d)]{Kuhn_Pereira}.

\begin{assumption}
Throughout this paper, $\capO$ will denote a reduced operad in $\Spectra$. We assume, for all $n \geq 0$, that $\capO[n]$ is $(-1)$-connected and that each map $I[n] \to \capO[n]$ of the unit morphism $I \to \capO$ is a flat stable cofibration between flat stable cofibrant objects in $\Spectra$. (Weaker than the $\Sigma$-cofibrancy condition needed for some of our results, this assumption guarantees that the forgetful functor $\AlgO \to \Spectra$ preserves cofibrant objects \cite[4.11]{Harper_Hess}.) Unless otherwise stated, we consider $\Spectra$ and $\AlgO$ with their positive flat stable model structures \cite[Section 7]{Harper_Hess}. 
\end{assumption}

\begin{rem}\label{rem_fibrant_o_algebras_are_semistable}
	Above, we have followed Schwede's terminology \cite{Schwede_book_project} for what Shipley refers to \cite{Shipley_commutative_ring_spectra} as the ``positive $S$-model structure.'' Relatedly, we make frequent use of the following fact to show that certain weak equivalences in $\AlgO$ forget to weak equivalences in $\mathsf{Sp}^\mathbb{N}$, i.e., induce isomorphisms on all (underived) stable homotopy groups. Suppose $X$ is an $\capO$-algebra that is fibrant in the positive flat stable model structure on $\AlgO$. Then the underlying symmetric spectrum $X$ is fibrant in the positive flat stable model structure on $\mathsf{Sp}^\Sigma$, and hence is also fibrant in the positive stable model structure on $\mathsf{Sp}^\Sigma$ (see \cite[Section 7]{Harper_Hess}). It follows from \cite[4.3]{Shipley_monoidal_uniqueness} that $X$ is a positive $\Omega$-spectrum, i.e., $X_n \to \Omega X_{n+1}$ is a weak equivalence for all $n > 0$ (see also \cite[14.2]{Mandell_May_Schwede_Shipley}). Hence, by \cite[4.2]{Schwede_homotopy_groups}, $X$ is semistable. Therefore, if $X \xrightarrow{f} Y$ is a weak equivalence of fibrant $\capO$-algebras in the positive flat stable model structure, then $f$ in fact induces isomorphisms on all (underived) stable homotopy groups, and so descends to a weak equivalence $\mathcal{U}X \xrightarrow{\mathcal{U}f} \mathcal{U}Y$ in $\mathsf{Sp}^\mathbb{N}$.
\end{rem}

\noindent
{\bf Acknowledgements.} The authors would like to thank John E. Harper, Duncan Clark, David White, and Yu Zhang for many useful conversations, and Mark Behrens and Doug Ravenel for several helpful correspondences. We are indebted to Nick Kuhn for alerting us to the comparison between the homotopy completion tower and the Taylor tower of the identity functor in the absence of connectivity assumptions, as well as an enlightening discussion of this point. The authors are also grateful to an anonymous referee who provided thorough and thoughtful feedback on an earlier version of this paper, which led to improved clarity of both our exposition and arguments. The second author was supported in part by Simons Foundation: Collaboration Grants for Mathematicians $\#$638247 and AFOSR grant FA9550-16-1-0212.\\

\section{Proof of Theorem \ref{thm_main_lemma_1}}\label{sec_proof_of_main_lemma_1}

The purpose of this section is to prove the first technical result of the paper, Theorem \ref{thm_main_lemma_1}. However, there are a few points that we need to address first. To begin, note that the statement of Theorem \ref{thm_main_lemma_1} implicitly uses the following fact. We suspect this is known to experts in the field (see, e.g., \cite[2.1]{Kuhn_Pereira}) but have included the statement and proof for the sake of completeness.

\begin{prop}
If $X$ is an $\capO$-algebra and $R$ is a commutative ring spectrum, then $R \wedge X$ inherits an $\capO$-algebra structure and the natural map $X \ \iso \ S\wedge X \to R \wedge X$ induced by the unit map of $R$ is a map of $\capO$-algebras.
\end{prop}
\begin{proof}
Let $\mu$ be the multiplication map of $R$ and $\capO\circ(X) \xrightarrow{\alpha} X$ the algebra structure map of $X$. The maps $\mu$ and $\alpha$ induce the following.
\begin{equation}\label{eqn_non_equivariant_commuting}
\begin{gathered}
\coprod_{n\geq 0} \capO[n]\wedge(R\wedge X)^{\wedge n} \ \iso \ \coprod_{n\geq 0} \capO[n] \wedge (R^{\wedge n} \wedge X^{\wedge n}) \xrightarrow{\mu_\ast} \\
\\
\xrightarrow{\mu_\ast} \coprod_{n \geq 0}\capO[n]\wedge(R\wedge X^{\wedge n}) \ \iso \ R\wedge \coprod_{n \geq 0} \capO[n] \wedge X^{\wedge n} \xrightarrow{\alpha_\ast} R\wedge X
\end{gathered}
\end{equation}
The fact that $R$ is (strictly) commutative implies that this composite is $\Sigma_n$-equivariant and hence induces a map
\begin{equation}
\capO\circ(R\wedge X) = \coprod_{n\geq 0} \capO[n]\wedge_{\Sigma_n}(R\wedge X)^{\wedge n} \to R \wedge X
\end{equation}
One then checks that this map is associative, unital, and compatible with the $\capO$-algebra structure on $X$.
\end{proof}

\begin{cor}\label{cor_commute_r_outside_o}
Given a commutative ring spectrum $R$, there is a natural map
\begin{align}
\capO\circ(R \wedge X) \to R \wedge \big( \capO\circ(X) \big)
\end{align}
\end{cor}
\begin{proof}
The composite of the first three ($\Sigma_n$-equivariant) maps of \eqref{eqn_non_equivariant_commuting} induces the desired map.
\end{proof}

Though discovered independently, a result similar to the following appears in unpublished work by John E. Harper in collaboration with the first author and Yu Zhang.

\begin{lem}\label{lem_commute_r_outside_o}
Suppose $\capO$ and $R$ are as in Theorem \ref{thm_main_lemma_1}. Let $Y$ be a cofibrant $\capO$-algebra and $R$ a commutative ring spectrum for which the pairing map $R \wedge R \weto R$ is a weak equivalence and which is cofibrant in $\Spectra$. Then the map in Corollary \ref{cor_commute_r_outside_o} is a weak equivalence.
\end{lem}
\begin{proof}
It follows from our assumptions and \cite[7.12]{Harper_Hess} that the multiplication map $R^{\wedge n} \weto R$ is a weak equivalence for any $n \geq 0$. Similarly, since $Y$ was assumed to be cofibrant, we know that $Y^{\wedge n}$ is as well and so the map $R^{\wedge n} \wedge Y^{\wedge n} \weto R\wedge Y^{\wedge n}$ is a weak equivalence. The assumption that $\capO$ is $\Sigma$-cofibrant implies that $\capO[n]$ is projectively cofibrant and so the functor $\capO[n] \wedge_{\Sigma_n}(-)$ preserves weak equivalences. By considering the injective stable model structure (see \cite[5.3]{Hovey_Shipley_Smith} and \cite[III.4.13]{Schwede_book_project}) on $\Spectra$, in which every object is cofibrant, it follows that the map
\begin{align}
\coprod_{n \geq 0} \capO[n] \wedge_{\Sigma_n}(R^{\wedge n} \wedge Y^{\wedge n}) \xrightarrow{\mu_\ast} \coprod_{n \geq 0}\capO[n]\wedge_{\Sigma n}(R \wedge Y^{\wedge n})
\end{align}
is a weak equivalence.
\end{proof}
\begin{rem}\label{rem_also_works_with_tn_o}
Above, the $\Sigma$-cofibrancy assumption on $\capO$ was used only to conclude that the functor $\capO[k] \wedge_{\Sigma_k}(-)$ preserves weak equivalences for all $k \geq 0$. For a fixed $n$, since $\tau_n\capO[k]$ is either $\capO[k]$ or trivial, it follows that $\tau_n\capO[k] \wedge_{\Sigma_k}(-)$ also preserves weak equivalences for all $k \geq 0$. The previous proof therefore remains valid if $\capO$ is replaced by $\tau_n\capO$.
\end{rem}

We can now give the proof of Theorem \ref{thm_main_lemma_1}.

\begin{proof}[Proof of Theorem \ref{thm_main_lemma_1}]
By replacing if necessary, it suffices to consider the case of a cofibrant $X \in \AlgO$. We construct a weak equivalence
\begin{align}\label{eqn_main_calcn}
\abs{\BAR(\tau_n\capO,\capO, R \wedge X)} \simeq R\wedge\abs{\BAR(\tau_n\capO,\capO,X)}
\end{align}
The proof is completed by appealing to \cite[4.10]{Harper_Hess}.

Our cofibrancy assumption on $\capO$ ensures that the forgetful functor $\AlgO \to \Spectra$ preserves cofibrant objects (see, e.g., \cite[4.11]{Harper_Hess}). Hence, the fact that $X$ is cofibrant implies that $\capO^{\circ k}\circ(X)$ is cofibrant for any $k \geq 0$ by  \cite[1.2]{Harper_bar_constructions}. Inductive application of Lemma \ref{lem_commute_r_outside_o}, with $Y = \capO^{\circ k} \circ (X)$, then shows that there is a levelwise weak equivalence
\begin{align}
\BAR(\tau_n\capO,\capO,R\wedge X) \weto R\wedge\BAR(\tau_n\capO,\capO,X)
\end{align}
Applying geometric realization and commuting with $R \wedge - $ completes the proof.
\end{proof}

\section{Constructing $S_E$ and the proof of Theorem \ref{thm_main_theorem}}\label{sec_convenient_model}

The purpose of this section is to construct a convenient model for the $E$-local sphere spectrum and prove the main result of the paper. In particular, we use \cite{Shipley_commutative_ring_spectra} and \cite{White_monoidal_bousfield_localizations} to find a symmetric spectrum $S_E$ that satisfies the desirable properties listed in Propositions \ref{prop_first_properties_of_s_e}. With this model in hand, we then prove Theorem \ref{thm_main_lemma_2} and conclude with the proof of Theorem \ref{thm_main_theorem}.

\begin{rem}\label{rem_subtlety}
There is a subtlety in this section that we wish to highlight. In his original work \cite{Bousfield_localization_spectra}, Bousfield did not have the luxury of a highly structured category of spectra in which to construct his localization functor $L_E$. His construction (see also \cite[II.9]{Schwede_book_project}) is therefore slightly different than the localization in the sense of \cite{Hirschhorn}, which we implicitly use in constructing $S_E$. We denote by $L$ the localization functor of \cite{Hirschhorn} in $\Spectra$ at the $E$-equivalences, i.e., $L$ is the composite of a fibrant replacement in the localized model structure composed with the identity ``back to'' the non-localized model structure. Using the comparisons established in \cite{Mandell_May_Schwede_Shipley}, one can show that $\mathcal{U}L(S)$ is a model for $L_E(\mathcal{U}S)$, and hence the two are canonically weakly equivalent in $\mathsf{Sp}^\mathbb{N}$.
\end{rem}

To construct $S_E$, first let $S^c \weto S$ be a functorial cofibrant replacement of the symmetric sphere spectrum $S$ in the model structure on commutative symmetric ring spectra established by \cite[3.2]{Shipley_commutative_ring_spectra}. The following two lemmas are needed to apply the relevant result of \cite{White_monoidal_bousfield_localizations} ultimately used to construct $S_E$.

\begin{lem}\label{lem_monoidal_localization}
$L$ is a monoidal Bousfield localization in the sense of \cite[4.4]{White_monoidal_bousfield_localizations}
\end{lem}
\begin{proof}
We appeal to \cite[4.6]{White_monoidal_bousfield_localizations}, noting that, when endowed with the positive flat stable model structure, the category $\Spectra$ is a cofibrantly generated monoidal model category in which cofibrant objects are flat; see \cite[3.1]{Shipley_commutative_ring_spectra} and \cite[7.12]{Harper_Hess}, respectively. 

Let $f \colon A \to B$ be an $E$-equivalence and $K$ a cofibrant symmetric spectrum. Cofibrantly replacing $A$ and $B$, we have that the map
\begin{align}
E \wedge A^c \weto E \wedge B^c
\end{align}
is a weak equivalence. Smashing with the cofibrant spectrum $K$ preserves this weak equivalence and it follows that the map
\begin{align}
E \wedge^L(A\wedge K) \weto E \wedge^L(B\wedge K)
\end{align}
is a weak equivalence, i.e., that $A\wedge K \to B \wedge K$ is an $E$-equivalence.
\end{proof}

\begin{lem}\label{lem_rectification_axiom}
The category of symmetric spectra with the positive flat stable model structure satisfies the rectification axiom of \cite[4.5]{White_model_structures_on_commutative_monoids}.
\end{lem}
\begin{proof}
This follows from \cite[3.3]{Shipley_commutative_ring_spectra} by taking ``$Y$'' to be the sphere spectrum. Alternatively, by considering symmetric sequences concentrated at $0$, the result follows from \cite[4.29*(b)]{Harper_symmetric_spectra_corrigendum}. 
\end{proof}

Lemmas \ref{lem_monoidal_localization} and \ref{lem_rectification_axiom} now show that the conditions of \cite[7.2]{White_monoidal_bousfield_localizations} are satisfied; hence, there is a commutative diagram
\begin{align}
\xymatrix{
S^c \ar ^-{(*)}[r] \ar[d] & \widetilde{S^c}\\
L(S^c) \ar _-{\sim}[ur]
}
\end{align}
in $\Spectra$, where $(*)$ is a map of commutative monoids. (Note that the existence of such a diagram is what it means for $L$ to preserve commutative monoids in the language of \cite{White_monoidal_bousfield_localizations}; see, for instance, the proof of \cite[3.2]{White_monoidal_bousfield_localizations}.) Replacing if necessary, we may also assume that $\widetilde{S^c}$ is fibrant.

\begin{defn}[Construction of $S_E$]
We define $S_E$ by factoring the map $(*)$ 
\begin{align}
\xymatrix{
S^c \ar[r] & S_E \ar ^-{\sim}[r] & \widetilde{S^c}
}
\end{align}
as a cofibration followed by an acyclic fibration in the model structure of \cite[3.2]{Shipley_commutative_ring_spectra}, namely the positive flat stable model structure on commutative ring spectra.

\end{defn}

The following Proposition details the advantageous properties of this construction.

\begin{prop}\label{prop_first_properties_of_s_e}
As defined, $S_E$ is
\begin{enumerate}[label=(\roman*)]
\item a strictly commutative monoid in $\Spectra$;
\item positive flat stable cofibrant and fibrant;
\item weakly equivalent to the $E$-local sphere $L(S)$; in fact, this equivalence is compatible with the localization map $S \to L(S)$ and the unit map $S \to S_E$ and furthermore induces an isomorphism on stable homotopy groups.
\end{enumerate}
\end{prop}
\begin{proof}
That $S_E$ is a strictly commutative monoid and fibrant is immediate from its construction. It follows from \cite[4.1]{Shipley_commutative_ring_spectra} that $S^c$ is positive flat stable cofibrant and so $S_E$ is as well. It is worth noting that, by \cite[7.12]{Harper_Hess}, the functor $S_E\wedge-$ preserves weak equivalences. 

To see that $S_E$ is naturally weakly equivalent to $L(S)$, consider the following commutative diagram
\begin{align}\label{diagram_natural_we_to_local_S}
\xymatrix{
S\ar@{=}[r] \ar@{=}[d]&
S\ar@{=}[r] \ar[d]&
S\ar@{=}[r] \ar[d]&
S \ar[d]\\
S\ar[d]&
S^c \ar _-{\sim}[l] \ar[r] \ar[d]&
S_E \ar ^-{\sim}[r] &
\widetilde{S^c}\\
L(S)&
L(S^c) \ar _-{\sim}[urr] \ar _-{(\#)}[l]
}
\end{align}
in $\Spectra$, where the upper vertical maps are the unit maps associated to each commutative monoid comprising the second row. The fact that $S^c \weto S$ is a weak equivalence implies that $(\#)$ is a weak equivalence as well, and this gives the zigzag of weak equivalences $S_E \simeq L(S)$. Furthermore, note that in the (bottom) zigzag of weak equivalences
\begin{align}\label{eqn_compatibility_of_we}
\xymatrix@C=1em{
&&&S\ar[dlll] \ar[dl] \ar[dr] \ar[drrr]\\
S_E \ar[rr]^-{\sim} &&  \widetilde{S^c} && L(S^c) \ar[ll]_-{\sim} \ar[rr]^-{\sim}&& L(S)
}
\end{align}
each object is fibrant, and so the weak equivalences in fact induce isomorphisms on stable homotopy groups (see Remark \ref{rem_fibrant_o_algebras_are_semistable}). Lastly, note that the leftmost downward arrow in \eqref{eqn_compatibility_of_we} is the unit map $S \to S_E$, while the rightmost downward arrow is the localization map $S \to L(S)$.
\end{proof}

\begin{proof}[Proof of Theorem \ref{thm_main_lemma_2}]
To begin, note that $S_E$ is cofibrant by Proposition \ref{prop_first_properties_of_s_e}. We now show the equivalence \eqref{eqn_good_model_localization} claimed in Theorem \ref{thm_main_lemma_2}. As shown in Proposition \ref{prop_first_properties_of_s_e}, the weak equivalence $S_E \simeq L(S)$ induces an isomorphism on stable homotopy groups and so remains a weak equivalence after forgetting to $\mathsf{Sp}^\mathbb{N}$. Together Remark \ref{rem_subtlety} and the fact that both $S_E$ and $L(S)$ are fibrant, this shows that the we have isomorphisms
\begin{align}
\mathcal{U}S_E \ \iso \ \mathcal{RU}S_E \ \iso \ \mathcal{RU}L(S) \ \iso \ \mathcal{U}L(S) \ \iso \ L_E(S)
\end{align}
in the stable homotopy category. 
Using the comparisons of \cite[0.3]{Mandell_May_Schwede_Shipley}, the fact that $L_E$ is smashing, and the naturality shown in \eqref{eqn_compatibility_of_we}, we then obtain a commutative diagram
\begin{align}\label{eqn_proof_of_good_model}
\xymatrix{
\mathcal{RU}X \ar[d] \ar ^-{\eta\wedge\id}[r] & \mathcal{RU}S_E\wedge_{\mathcal{SHC}} \mathcal{RU}X \ar ^-{\iso}[dl]\\
L_E(\mathcal{RU}X)
}
\end{align}
in the stable homotopy category, where $\eta$ is induced by the unit map $S \to S_E$ and the vertical map is the natural localization map \cite{Bousfield_localization_spectra} of $\mathcal{RU}X$. This establishes \eqref{eqn_good_model_localization}.

Lastly, we show that the pairing map $S_E \wedge S_E \weto S_E$ is a weak equivalence. Indeed, by taking $X = S_E$ in \eqref{eqn_proof_of_good_model}, the ``two-out-of-three'' property implies that the unit map $\mathcal{RU}S_E \to \mathcal{RU}S_E \wedge_{\mathcal{SHC}} \mathcal{RU}S_E$ is an isomorphism in $\Ho(\mathsf{Sp}^\mathbb{N})$. It follows from \cite[0.3]{Mandell_May_Schwede_Shipley} that the corresponding map $S_E \weto S_E \wedge S_E$ is a weak equivalence in $\Spectra$, and the fact that there is a retract
\begin{align}
\xymatrix{
S_E \ar[r] \ar _-{\id}[dr] & S_E \wedge S_E \ar[d]\\
& S_E
}
\end{align} 
in $\Spectra$ from the commutative monoid structure of $S_E$ then completes the proof.
\end{proof}

\begin{proof}[Proof of Theorem \ref{thm_main_theorem}]
It follows from Theorem \ref{thm_main_lemma_1} that there is a weak equivalence
\begin{align}\label{eqn_proof_of_main_theorem_1}
\mathcal{RU}\big(S_E \wedge (\tau_n\capO\circ_\capO^\mathsf{h} (X))\big) \simeq \mathcal{RU}\big( \tau_n\capO\circ_\capO^\mathsf{h}(S_E\wedge X)\big)
\end{align}
in $\mathsf{Sp}^\mathbb{N}$. By \cite[0.3]{Mandell_May_Schwede_Shipley} and Theorem \ref{thm_main_lemma_2}, the left hand side of \eqref{eqn_proof_of_main_theorem_1} is weakly equivalent in $\mathsf{Sp}^\mathbb{N}$ to $L_E\big(\mathcal{RU}(\tau_n\capO\circ_\capO^\mathsf{h}(X))\big)$ which, since $\tau_n\capO\circ_\capO^\mathsf{h}(X)$ is fibrant, is weakly equivalent in $\mathsf{Sp}^\mathbb{N}$ to $L_E(\mathcal{U}\tau_n\capO\circ_\capO^\mathsf{h}(X))$. For the same fibrancy reason, the right hand side of \eqref{eqn_proof_of_main_theorem_1} is weakly equivalent in $\mathsf{Sp}^\mathbb{N}$ to $\mathcal{U}\tau_n\capO\circ_\capO^\mathsf{h}(S_E\wedge X)$. This establishes the first weak equivalence of Theorem \ref{thm_main_theorem}.

To see that $\mathcal{U}(S_E\wedge X) \simeq L_E(\mathcal{U}X)$, recall that both $S_E$ and $X$ are fibrant, and that the former is also cofibrant. By \cite[4.10]{Schwede_homotopy_groups}, $S_E \wedge X$ is semistable (see \cite[5.6.1]{Hovey_Shipley_Smith}). It follows that $\mathcal{U}(S_E \wedge X)$ is weakly equivalent in $\mathsf{Sp}^\mathbb{N}$ to $\mathcal{RU}(S_E \wedge X)$, the latter of which is weakly equivalent in $\mathsf{Sp}^\mathbb{N}$ to $L_E(\mathcal{RU}X)$. Since $X$ is fibrant, $L_E(\mathcal{RU}X) \simeq L_E(\mathcal{U}X)$ in $\mathsf{Sp}^\mathbb{N}$, and this completes the proof.
\end{proof}

\section{Identification of the Taylor tower}\label{sec_identification_taylor_tower}
The purpose of this section is to show the following, which identifies the homotopy completion tower with the Taylor tower of the identity functor on $\capO$-algebras and, along with Theorem \ref{thm_main_theorem}, proves Theorem \ref{thm_main_corollary}. Note that, in this section, we continue to assume $\capO$ is $\Sigma$-cofibrant.

\begin{thm}\label{thm_identification_of_taylor_tower}
For any cofibrant $\capO$-algebra $X$, there is a weak equivalence
\begin{align}
P_n(\id)(X) \simeq \tau_n\capO\circ_\capO(X)
\end{align}
in $\AlgO$.
\end{thm}

A proof of this result appears in \cite[2.21]{Kuhn_Pereira}, and is based on an argument given by Pereira in \cite[4.3]{Pereira_spectral_operad}. As \cite{Pereira_spectral_operad} has not yet been published, we have included an alternative proof for the sake of completeness. The authors would like to thank Nick Kuhn for outlining this different strategy \cite{Kuhn_private_communication}.

\begin{rem}
It is worth noting that, if $X$ is $0$-connected, Theorem \ref{thm_identification_of_taylor_tower} can be obtained as a consequence of the connectivity estimates used to prove \cite[1.12]{Harper_Hess}. We emphasize that Theorem \ref{thm_identification_of_taylor_tower} makes no connectivity assumptions.
\end{rem}

To keep this section appropriately brief, we assume the reader is familiar with standard constructions in Goodwillie calculus, both in the context of spaces and $\capO$-algebras (see, for instance, \cite{Goodwillie_calculus_1}, \cite{Goodwillie_calculus_2}, \cite{Goodwillie_calculus_3}, \cite{Pereira_general_context}, and \cite{Pereira_spectral_operad}). In particular, when working in $\capO$-algebras, one often implicitly (pre)composes with functorial (co)fibrant replacements to keep things homotopically meaningful.

The strategy of our proof is as follows. To begin, because $P_n$ behaves particularly well when applied to spectrum-valued functors, it is advantageous to reduce Theorem \ref{thm_identification_of_taylor_tower} to proving the result when we consider the functors $\id$ and $\tau_n\capO\circ_\capO(-)$ as landing in $\Spectra$, which is accomplished by Lemma \ref{lem_u_commutes_p_n}. Next, we consider the free-forgetful adjunction 
\begin{align}
\xymatrix{
\Spectra \ar@<0.5ex>^-{\capO\circ(-)}[r] & \AlgO\ar@<0.5ex>^-{U}[l]
}
\end{align}
and analyze the $n^{th}$ degree Taylor approximation to the free $\capO$-algebra functor in Lemma \ref{lem_calc_of_approx_of_free}. Lastly, to prove Theorem \ref{thm_identification_of_taylor_tower}, we use the fact \cite[1.8]{Harper_bar_constructions} that the identity functor on $\AlgO$ can be resolved as the homotopy colimit of iterates of the free $\capO$-algebra functor.

\begin{rem}
It is common to apply $\capO\circ(-)$ to an $\capO$-algebra, in which case one is implicitly precomposing with the forgetful functor $U$.
\end{rem}

\begin{lem}\label{lem_u_commutes_p_n}
Suppose $F$ and $G$ are homotopy functors defined on, and with values in, $\capO$-algebras. If a map $F \to G$ induces an equivalence $P_n(UF) \weto P_n(UG)$, then it also induces an equivalence $P_nF \weto P_nG$.
\end{lem}
\begin{proof}
Since the positive flat stable model structure on $\AlgO$ is induced by the forgetful functor $U$, we know that homotopy limits in $\AlgO$ are calculated in the underlying category $\Spectra$, and that $U$ reflects weak equivalences. It follows from \cite[3.27]{Harper_symmetric_spectra} and \cite[4.11]{Harper_Hess} that filtered homotopy colimits in $\AlgO$ are also calculated in the underlying category of spectra. In particular, this means that the objects of homotopy limits and filtered homotopy colimits in $\AlgO$ are calculated in $\Spectra$, and the $\capO$-algebra structure on that object is then induced by the naturality of the homotopy (co)limit construction. Hence, the forgetful functor $U$ commutes naturally with the construction of $P_n$, i.e., we have a natural equivalence $UP_n \simeq P_nU$.
\end{proof}

\begin{lem}\label{lem_calc_of_approx_of_free}
Consider the map of operads $\capO \to \tau_n\capO$. Applying $P_n$ to the induced map $U\capO \circ (-) \to U\tau_n\capO\circ(-)$ of endofunctors on $\Spectra$ yeilds an equivalence of functors $P_n\big(U\capO\circ(-)\big) \weto P_n\big(U\tau_n\capO\circ(-)\big)$.
\end{lem}
\begin{proof}
For any $X$ in $\Spectra$, note that the following diagram
\begin{align}
\xymatrix{
\coprod_{k=0}^n\capO[k]\wedge_{\Sigma_k}X^{\wedge k} \ar^-{(\ast)}[r] \ar_-{\id}[dr] &
\coprod_{k=0}^\infty\capO[k]\wedge_{\Sigma_k}X^{\wedge k} \ar^-{(\#)}[d]\\ 
& \coprod_{k=0}^n\capO[k]\wedge_{\Sigma_k}X^{\wedge k} 
}
\end{align}
in $\Spectra$ commutes, where $(*)$ is the canonical inclusion and $(\#)$ is induced by mapping $\capO[k]$ to a point for $k > n$, i.e., $(\#)$ is the map induced by $U\capO\circ(-) \to U\tau_n\capO\circ(-)$. By the ``two-out-of-three'' property, it will therefore suffice to show that applying $P_n$ to $(*)$ yields a weak equivalence. Towards that end, note that since $P_n$ naturally commutes with sequential homotopy colimits, we have weak equivalences of functors as below.

\begin{align}\label{eqn_pn_hocolim}
\begin{gathered}
P_n(\coprod_{k=0}^\infty\capO[k]\wedge_{\Sigma_k}(-)^{\wedge k}) \simeq P_n\big(\hocolim_m(\coprod_{k=0}^m\capO[k]\wedge_{\Sigma_k}(-)^k)\big)\\
\simeq \hocolim_mP_n(\coprod_{k=0}^m\capO[k]\wedge_{\Sigma_k}(-)^k)
\end{gathered}
\end{align}
(In fact, for spectrum-valued functors, $P_n$ commutes with all homotopy colimits; see \cite[1.7]{Goodwillie_calculus_3}.)

To analyze \eqref{eqn_pn_hocolim}, let us fix $m \geq n$. It follows from \cite[3.1]{Goodwillie_calculus_3} (or \cite[5.24]{Pereira_general_context}) that the functor 
\begin{align}
X \mapsto \capO[k]\wedge_{\Sigma_k} X^{\wedge k}
\end{align}
in $\Spectra$ is $k$-homogeneous. Together with the facts that (i) $P_n$ commutes with homotopy limits and (ii) finite coproducts and products agree in $\Spectra$, this shows that 
\begin{align}
P_n(\coprod_{k=0}^m\capO[k]\wedge_{\Sigma_k}(-)^k) \simeq P_n(\prod_{k=0}^m\capO[k]\wedge_{\Sigma_k}(-)^k) \simeq P_n(\coprod_{k=0}^n\capO[k]\wedge_{\Sigma_k}(-)^k)
\end{align}
Hence, all of the structure maps in the homotopy colimit in \eqref{eqn_pn_hocolim} are weak equivalences after the $m=n$ term. It follows that the map of functors 
\begin{align}
\coprod_{k=0}^n\capO[k]\wedge_{\Sigma_k}(-)^k \longrightarrow \coprod_{k=0}^\infty\capO[k]\wedge_{\Sigma_k}(-)^k
\end{align}
on $\Spectra$ induced by the canonical inclusion yields an equivalence after applying $P_n$, which suffices to complete the proof.
\end{proof}

\begin{proof}[Proof of \ref{thm_identification_of_taylor_tower}]
It is explicitly shown in \cite[2.21]{Kuhn_Pereira} that $\tau_n\capO\circ_\capO(-)$ is $n$-excisive. The desired identification $P_n(\id)\simeq\tau_n\capO\circ_\capO(-)$ will therefore follow by showing an equivalence
\begin{align}\label{eqn_want_to_show}
P_n(\id) \simeq P_n\big(\tau_n\capO\circ_\capO(-)\big)
\end{align}
of functors $\AlgO \to \AlgO$.

The strategy now is to resolve the identity functor by iterates of the free $\capO$-algebra functor $\capO\circ(-)$. Crucial to the following argument are three of the main results of \cite{Harper_bar_constructions}. Indeed, it follows from \cite[1.8]{Harper_bar_constructions} that we have a weak equivalence of functors 
\begin{align}
\id\simeq \hocolim\BAR(\capO,\capO,-) \colon \AlgO \to \AlgO
\end{align}
This equivalence holds after composing with the forgetful functor and applying $P_n$, i.e.,
\begin{align}
P_n(U)\simeq P_n\big( U\hocolim\BAR(\capO,\capO,-) \big)
\end{align}
We then have the following weak equivalences of functors, where the constructions in simplicial $\capO$-algebras are done levelwise.
\begin{equation}\label{eqn_big_zigzag}
\begin{split}
&P_n(U)\simeq P_n\big(U\hocolim\BAR(\capO,\capO,-) \big) \overset{(1)}{\simeq} P_n\big(\hocolim U\BAR(\capO,\capO,-) \big)\\
&\overset{(2)}{\simeq} \hocolim P_nU\BAR(\capO,\capO,-) \overset{(3)}{\simeq} \hocolim P_nU\BAR(\tau_n\capO,\capO,-)\\
& \overset{(4)}{\simeq} P_n\hocolim U \BAR(\tau_n\capO,\capO,-) \overset{(5)}{\simeq} P_n\big( U\hocolim\BAR(\tau_n\capO,\capO,-) \big) \\
&\overset{(6)}{\simeq} P_n\big(U\tau_n\capO\circ_\capO(-)\big)
\end{split}
\end{equation}
Equivalences $(1)$ and $(5)$ follow from the fact \cite[1.6]{Harper_bar_constructions} that homotopy colimits of simplicial diagrams in $\AlgO$ can be calculated in the underlying category of spectra, and hence naturally commute with the forgetful functor $U$.

Next, equivalences $(2)$ and $(4)$ follow from the fact \cite[1.7]{Goodwillie_calculus_3} that $P_n$ commutes with homotopy colimits of spectrum-valued functors, essentially because $\hocolim$'s of spectra preserve cartesian cubes. Equivalence $(6)$ follows from the identification
\begin{equation}
\hocolim\BAR(\tau_n\capO,\capO,-)\simeq \tau_n\capO\circ_\capO(-)
\end{equation}
given by \cite[1.10]{Harper_bar_constructions}. Lastly, equivalence $(3)$ is derived as follows. 

Consider the map $U\BAR(\capO,\capO,0) \to U\BAR(\tau_n\capO,\capO,-)$ induced by $\capO \to \tau_n\capO$. Applying $P_n$ at simplicial degree one, for instance, gives the top row of the following diagram.
\begin{align}
\xymatrix{
P_n(U\capO\circ\capO\circ-) \ar[d] \ar[r] &  P_n(U\tau_n\capO\circ\capO\circ-)\ar[d]\\
P_n\big(P_n(U\capO\circ-)\circ\capO\circ-\big) \ar[r] & P_n\big(P_n(U\tau_n\capO\circ-)\circ\capO\circ-\big)
}
\end{align}

Lemma \ref{lem_calc_of_approx_of_free} shows that the bottom row of the diagram above is a weak equivalence, while a straightforward generalization of \cite[3.1]{Arone_Ching} to $\capO$-algebras (as detailed in \cite[4.9]{Pereira_spectral_operad}) shows that the vertical arrows are weak equivalences. Hence, the top row is a weak equivalence. Performing a similar analysis at each simplicial degree gives equivalence (3).

To complete the proof, it is a straightforward check that the zigzag of \eqref{eqn_big_zigzag} is compatible with the natural map $\id \ \iso \ \capO\circ_\capO(-) \to \tau_n\capO\circ_\capO(-)$ induced by the map of operads $\capO \to \tau_n\capO$. Lemma \ref{lem_u_commutes_p_n} then gives the desired equivalence of \eqref{eqn_want_to_show}, completing the proof.
\end{proof}

\bibliographystyle{plain}
\bibliography{aE_n_localization_bibliography}

\begin{thebibliography}{10}

\bibitem{Arone_Ching}
G.~Arone and M.~Ching.
\newblock Operads and chain rules for the calculus of functors.
\newblock {\em Ast\'erisque}, (338):vi+158, 2011.

\bibitem{Bousfield_localization_spectra}
A.~K. Bousfield.
\newblock The localization of spectra with respect to homology.
\newblock {\em Topology}, 18(4):257--281, 1979.

\bibitem{Ching_Harper_derived_Koszul_duality}
M.~Ching and J.~E. Harper.
\newblock Derived {K}oszul duality and {TQ}-homology completion of structured
  ring spectra.
\newblock {\em Adv. Math.}, 341:118--187, 2019.

\bibitem{Dwyer_Spalinski}
W.~G. Dwyer and J.~Spali{\'n}ski.
\newblock Homotopy theories and model categories.
\newblock In {\em Handbook of algebraic topology}, pages 73--126.
  North-Holland, Amsterdam, 1995.

\bibitem{EKMM}
A.~D. Elmendorf, I.~Kriz, M.~A. Mandell, and J.~P. May.
\newblock {\em Rings, modules, and algebras in stable homotopy theory},
  volume~47 of {\em Mathematical Surveys and Monographs}.
\newblock American Mathematical Society, Providence, RI, 1997.
\newblock With an appendix by M. Cole.

\bibitem{Goodwillie_calculus_1}
T.~G. Goodwillie.
\newblock Calculus. {I}. {T}he first derivative of pseudoisotopy theory.
\newblock {\em $K$-Theory}, 4(1):1--27, 1990.

\bibitem{Goodwillie_calculus_2}
T.~G. Goodwillie.
\newblock Calculus. {II}. {A}nalytic functors.
\newblock {\em $K$-Theory}, 5(4):295--332, 1991/92.

\bibitem{Goodwillie_calculus_3}
T.~G. Goodwillie.
\newblock Calculus. {III}. {T}aylor series.
\newblock {\em Geom. Topol.}, 7:645--711, 2003.

\bibitem{Harper_symmetric_spectra}
J.~E. Harper.
\newblock Homotopy theory of modules over operads in symmetric spectra.
\newblock {\em Algebr. Geom. Topol.}, 9(3):1637--1680, 2009.
\newblock Corrigendum: \emph{Algebr. Geom. Topol.}, 15(2):1229--1237, 2015.

\bibitem{Harper_bar_constructions}
J.~E. Harper.
\newblock Bar constructions and {Q}uillen homology of modules over operads.
\newblock {\em Algebr. Geom. Topol.}, 10(1):87--136, 2010.

\bibitem{Harper_symmetric_spectra_corrigendum}
J.~E. Harper.
\newblock Corrigendum to ``{H}omotopy theory of modules over operads in
  symmetric spectra'' [ {MR}2539191].
\newblock {\em Algebr. Geom. Topol.}, 15(2):1229--1237, 2015.

\bibitem{Harper_Hess}
J.~E. Harper and K.~Hess.
\newblock Homotopy completion and topological {Q}uillen homology of structured
  ring spectra.
\newblock {\em Geom. Topol.}, 17(3):1325--1416, 2013.

\bibitem{Hirschhorn}
P.~S. Hirschhorn.
\newblock {\em Model categories and their localizations}, volume~99 of {\em
  Mathematical Surveys and Monographs}.
\newblock American Mathematical Society, Providence, RI, 2003.

\bibitem{Hovey_Shipley_Smith}
M.~Hovey, B.~Shipley, and J.~H. Smith.
\newblock Symmetric spectra.
\newblock {\em J. Amer. Math. Soc.}, 13(1):149--208, 2000.

\bibitem{Johnson_Wilson}
David~Copeland Johnson and W.~Stephen Wilson.
\newblock {$BP$} operations and {M}orava's extraordinary {$K$}-theories.
\newblock {\em Math. Z.}, 144(1):55--75, 1975.

\bibitem{Kuhn_private_communication}
N.~Kuhn.
\newblock Private communication, 2020.

\bibitem{Kuhn_Pereira}
N.~J. Kuhn and L.~A. Pereira.
\newblock Operad bimodules and composition products on {A}ndr\'e-{Q}uillen
  filtrations of algebras.
\newblock {\em Algebr. Geom. Topol.}, 17(2):1105--1130, 2017.

\bibitem{Mandell_May_Schwede_Shipley}
M.~A. Mandell, J.~P. May, S.~Schwede, and B.~Shipley.
\newblock Model categories of diagram spectra.
\newblock {\em Proc. London Math. Soc. (3)}, 82(2):441--512, 2001.

\bibitem{Pereira_general_context}
L.~A. Pereira.
\newblock A general context for {G}oodwillie calculus.
\newblock {\em \verb=arXiv:1301.2832 [math.AT]=}, 2013.

\bibitem{Pereira_spectral_operad}
L.~A. Pereira.
\newblock {G}oodwillie calculus in the category of algebras over a spectral
  operad.
\newblock 2013.
\newblock \\ Available at:
  \verb=https://services.math.duke.edu/~lpereira/index.html=.

\bibitem{Ravenel_nilpotence}
D.~C. Ravenel.
\newblock {\em Nilpotence and periodicity in stable homotopy theory}.
\newblock Princeton University Press, 1992.

\bibitem{Schonsheck_fibration_theorems}
N.~Schonsheck.
\newblock Fibration theorems for {$\mathsf{TQ}$}-completion of structured ring
  spectra.
\newblock {\em Tbilisi Math. J., Special Issue on Homotopy Theory, Spectra, and
  Structured Ring Spectra}, 2020.

\bibitem{Schonsheck_taylor_tower_identity}
N.~Schonsheck.
\newblock {$\TQ$}-completion and the {T}aylor tower of the identity functor.
\newblock {\em J. Homotopy Relat. Struct.}, 17:201--216, 2022.

\bibitem{Schwede_homotopy_groups}
S.~Schwede.
\newblock On the homotopy groups of symmetric spectra.
\newblock {\em Geom. Topol.}, 12(3):1313--1344, 2008.

\bibitem{Schwede_book_project}
S.~Schwede.
\newblock Symmetric spectra.
\newblock 2012.
\newblock Available at:\\
  \verb=http://www.math.uni-bonn.de/people/schwede/SymSpec-v3.pdf=.

\bibitem{Shipley_commutative_ring_spectra}
B.~Shipley.
\newblock A convenient model category for commutative ring spectra.
\newblock In {\em Homotopy theory: relations with algebraic geometry, group
  cohomology, and algebraic $K$-theory}, volume 346 of {\em Contemp. Math.},
  pages 473--483. Amer. Math. Soc., Providence, RI, 2004.

\bibitem{Shipley_monoidal_uniqueness}
Brooke Shipley.
\newblock Monoidal uniqueness of stable homotopy theory.
\newblock {\em Adv. Math.}, 160(2):217--240, 2001.

\bibitem{White_monoidal_bousfield_localizations}
D.~White.
\newblock Monoidal {B}ousfield localizations and algebras over operads.
\newblock {\em Accepted to Proceedings of the JPC Greenlees Conference}, 2020.
\newblock Available at {\verb=arXiv:1404.5197 [math.AT]=}.

\bibitem{White_model_structures_on_commutative_monoids}
David White.
\newblock Model structures on commutative monoids in general model categories.
\newblock {\em J. Pure Appl. Algebra}, 221(12):3124--3168, 2017.

\end{thebibliography}

\end{document}